\documentclass{amsart}

\newtheorem{theorem}{Theorem}[section]
\newtheorem{lemma}[theorem]{Lemma}
\newtheorem{proposition}[theorem]{Proposition}
\newtheorem{corollary}[theorem]{Corollary}

\theoremstyle{definition}
\newtheorem{definition}[theorem]{Definition}

\theoremstyle{remark}
\newtheorem{remark}[theorem]{Remark}

\numberwithin{equation}{section}

\begin{document}

\title{Dynamics of Orthonormal Bases Associated to Basins of Attraction}

\author{James Tipton}
\address{Department of Mathematical Sciences, 101 Mathematics Building, University of Montana, Missoula, 59812}
\email{james.tipton@mso.umt.edu}


\date{December 23, 2016}


\begin{abstract}
In \cite{alpay2015infinite}, a technique was developed which allows for the construction of a reproducing kernel Hilbert space on basins of attraction containing $0$.  When the right conditions are met, an explicit orthonormal basis can be constructed using a particular class of operators.  It is natural then to consider how the orthonormal basis changes as we let the basin of attraction vary.  We will consider this question for the basins of attraction containing $0$ of the family of polynomials $\mathcal{F} = \{az^{2^{n+2}}-2az^{2^{n+1}}:a\neq0\}$, where $n\in\mathbb{N}$.
\end{abstract}

\maketitle

\tableofcontents

\section{Introduction}
If $R$ is a rational map, then associated to any attracting fixed point, $\zeta$, of $R$ is the basin of attraction containing $\zeta$, which we will denote $B_{R,\zeta}$.  Following the approach of \cite{alpay2015infinite}, it was shown in \cite{jet2016} that whenever $R$ is a polynomial with an attracting fixed point at $0$, one obtains a kernel function of the form
\[K(z,w) = \prod_{i=0}^\infty\left(1+\left(R^{\circ i}(z)\overline{R^{\circ i}(w)}\right)^\alpha\right)\]

\noindent for any positive integer $\alpha$, and where $R^{\circ i}$ is shorthand for $R$ composed with itself $i$ times.  Associated to each kernel function is a unique reproducing kernel Hilbert space.  On the reproducing kernel Hilbert spaces of the kernel functions above, we have the following operators:
\[S_0(f)(z) = f(R(z))\quad\text{ and }\quad S_1(f)(z) = z^\alpha f(R(z))\]

\noindent Following the work in \cite{alpay2015infinite}, if these operators satisfy the Cuntz relations,
\[S^{*}_iS_j = \delta_{ij}I\quad\text{ and}\quad\sum S_iS^{*}_i = I,\]

\noindent then the functions given by $S_{i_1}\cdots S_{i_m}\mathbf{1}(z)$, where $m\in\mathbb{N}$, $i_j\in\{0,1\}$, and $\mathbf{1}(z)=1$, form an orthonormal basis for the reproducing kernel Hilbert space associated to the kernel function $K(z,w)$.  It was shown in \cite{jet2016} that this is the case for the family of polynomials $\mathcal{F} = \{az^{2^{n+2}}-2az^{2^{n+1}}:a\neq0\}$, if we take the exponent $\alpha$ in $K(z,w)$ to be $2^n$.  We will show that the corresponding orthonormal bases vary continuously on the parameter $a$.\

\section{Basic Definitions and Propositions}

\subsection{Complex Dynamics}
Two excellent references for a more detailed and complete theory of complex dynamics include \cite{beardon2000iteration} and \cite{milnor2006dynamics}.  We will always denote by $R$ and $S$ a rational map.

\begin{definition}
If $\zeta$ is a complex number satisfying the properties $R(\zeta) = \zeta$ and $|R^{'}(\zeta)|<1$, then we call $\zeta$ an attracting fixed point of $R$.  In particular, if $R^{'}(\zeta) = 0$, then $\zeta$ is called a super-attracting fixed point of $R$.
\end{definition}

\noindent Associated to any attracting fixed point is the set of those complex numbers which are ``attracted'' to the fixed point.

\begin{definition}
Suppose $R$ has an attracting fixed point $\zeta$.  The basin of attraction of $R$ containing the attracting fixed point $\zeta$ is given by

\[B_{R,\zeta} = \left\{z\in\mathbb{C}:\lim_{i\rightarrow\infty}R^{\circ i}(z)=\zeta\right\}\]

\noindent where $R^{\circ i}$ denotes $R$ composed with itself $i$ times.
\end{definition}

Some basins of attraction at $0$ have a seemingly uncommon characterization:

\begin{lemma}
Suppose $R$ is a rational map with an attracting fixed point at $0$.  If $\infty$ does not belong to $B_{R,0}$, then
\[B_{R,0} = \left\{z\in\mathbb{C}:\sum_{i=0}^\infty |R^{\circ i}(z)|<\infty\right\}\]

\end{lemma}

\begin{proof}
We first note that any such degree $n$ rational map must be of the form

\[R(z) = \dfrac{\sum_{i=1}^na_iz^i}{\sum_{i=0}^mb_iz^i},\qquad \left|\dfrac{a_1}{b_0}\right|<1,\quad a_n,b_m,b_0\neq0\]

\begin{itemize}
\item\noindent\underline{$\mathbf{\Omega\subseteq B_{R,0}}$}\\

If $z\in\Omega$, then $\lim_{i\rightarrow\infty}R^{i}(z) = 0$, and thus $\Omega$ is contained within $B_{R,0}$.\\

\item\noindent\underline{$\mathbf{B_{R,0}\subseteq\Omega}$}\\

To obtain the reverse inclusion, we use the ratio test.  Let $\zeta\in B_{R,0}$ and let $z_j = R^{\circ j}(\zeta)$.  We then compute

\[\lim_{j\rightarrow\infty}\dfrac{|R^{\circ j+1}(\zeta)|}{|R^{\circ j}(\zeta)|}=\lim_{j\rightarrow\infty}\dfrac{|R(z_j)|}{|z_j|} = \lim_{n\rightarrow\infty}\dfrac{|\sum_{i=0}^{n-1}a_{i+1}z_j^i|}{|\sum_{i=0}^mb_iz_j^i|}=\left|\dfrac{a_1}{b_0}\right|<1\]

\noindent Thus we have that the series

\[\sum_{i=0}^\infty|R^{\circ i}(\zeta)|\]

\noindent converges absolutely, giving $\zeta\in\Omega$, which completes the proof.

\end{itemize}
\end{proof}

\begin{remark}
When $0<R^{'}(0)<1$, the above lemma easily follows from the proof of Koenigs' Linearization Theorem.  However, through the ratio test, one can see that the lemma actually holds for $R^{'}(0) = 0$ as well.  The usefulness of this lemma will become apparent when combined Lemma 3.2.
\end{remark}





\subsection{Reproducing Kernel Hilbert Spaces}

A classic reference for the theory of kernel functions is \cite{aronszajn1950theory}, while a more modern approach is given in \cite{paulsen2016introduction}.

\begin{definition}
Suppose $X$ is a set.  A reproducing kernel Hilbert space on $X$ is a Hilbert space of functions, $\mathcal{H}$, on $X$ for which every linear evaluation functional is bounded.
\end{definition}

\noindent From this one is able to conclude the existence of a unique function on $X$, $K(z,w)$ satisfying
\[h(w) = \langle h(z), K(z,w) \rangle\]

\noindent for all $h\in\mathcal{H}$.  This function has many names, including kernel function, reproducing kernel, and positive definite function.  It can be shown that there is a one-to-one correspondence between the collection of all positive definite functions on a set $X$ and the collection of all reproducing kernel Hilbert spaces on the set $X$.  For this reason, $(\mathcal{H}, K)$ is often referred to as a reproducing kernel pair on $X$.  Reproducing kernel Hilbert spaces have many nice properties; among them is the following:

\begin{theorem}
Suppose $(\mathcal{H},K)$ is a reproducing pair on some set $X$.  If $\{e_i\}_{i\in I}$ is an orthonormal basis of $\mathcal{H}$, then
\[K(z,w) = \sum_{i\in I}e_i(z)\overline{e_i(w)}\]
\end{theorem}

\noindent Essentially, if one knows an orthonormal basis for $\mathcal{H}$, then one knows the kernel function $K$ as well.

\section{Kernel Functions on Basins of Attraction}

 In this section, we recall results from \cite{alpay2015infinite} which are relevant to later sections.  In what follows, we will let $\mathcal{U}$ be a topological space, $(\mathcal{H}_k, k)$ a reproducing kernel pair on $\mathcal{U}$ with orthonormal basis $\{e_i\}_{i\in I}$, and $R$ an endomorphism on $\mathcal{U}$ such that there exists $l\in\mathcal{U}$ satisfying $\lim\limits_{n\rightarrow\infty}R_n(z) = l$ for all $z\in\mathcal{U}$.  The following proposition says that under certain conditions one may obtain a kernel function on $\mathcal{U}$ as an infinite product involving the iteration of the endomorphism $R$.

\begin{proposition}
Suppose that $K(z,w)$ is a continuous map on $\mathcal{U}$, not identically $0$, such that $K(z,w) = k(z,w)K(R(z), R(w))$.  If $K(l,l)>0$, then $K(z,w)$ is a kernel function on $\mathcal{U}$ and 
\[K(z,w) = \left(\prod_{n=0}^\infty\left(\sum_{i\in I}e_i(R_n(z))\overline{e_i(R_n(w))}\right)\right)K(l,l)\]
\end{proposition}

\noindent The following lemma will allow us to apply Proposition 3.1 to basins of attraction at $0$.

\begin{lemma}
If $k(z,w) = 1 + t(z,w)$ where $t$ is also a kernel function on $\mathcal{U}$ and the set
\[\Omega = \left\{z\in\mathcal{U}: \sum_{n=0}^\infty|t(R_n(z),R_n(z)|<\infty\right\}\]
\noindent is non-empty, then the infinite product
\[K(z,w) = \prod_{n=0}^\infty k(R_n(z),R_n(w))\]
\noindent converges on $\Omega$, and satisfies $K(z,w) = k(z,w)K(R(z), R(w))$ for all $z,w\in\Omega$.
\end{lemma}

The next few result aim for an explicit construction of an orthonormal basis for the reproducing kernel Hilbert space $\mathcal{H}$ associated to the kernel function $K$ obtained in Proposition 3.1.  Suppose that for every $z\in\Omega$, the map $R$ satisfies the cardinality condition
\[M(z) = \text{Card}\{\zeta\in\Omega:\,R(\zeta) = z\}<\infty\]

\noindent and either
\[\dfrac{1}{M(z)}\sum_{R(\zeta)=z}e_i(\zeta)e_j(\zeta)^{*} = \delta_{ij},\qquad \forall i,j\in I \tag{$\dagger$}\]

\noindent or
\[\dfrac{1}{M(z)}\sum_{R(\zeta)=z}e_i(\zeta)e_j(\zeta) = \delta_{ij},\qquad \forall i,j\in I \tag{$\ddagger$}\]

\noindent These conditions were designed with the Cuntz relations in mind.  Details on these relations can be found in \cite{cuntz1977simplec}.

\begin{theorem} 
Consider the operators $\{S_i\}_{i\in I}$ defined on $\mathcal{H}$ by $S_i(f)(z) = e_i(z)f(R(z))$. If either $\dagger$ or $\ddagger$ holds, then the operators $\{S_i\}_{i\in I}$ satisfy the Cuntz relations.
\end{theorem}

With these operators one can construct an orthonormal basis of $\mathcal{H}$.  First observe that the function $\mathbf{1}(z) = 1$ is in $\mathcal{H}$.  Let $\rho = \dim\mathcal{H}_k$ and let $J$ be an index set with cardinality equal to $\rho$.  For each natural number $N$, let $S_{\iota_N} = S_{j_1}\cdots S_{j_N}$, where $\iota_N=(j_1,\dots,j_N)\in J^N$.  Next define an index set $J^\infty$ by
\[J^\infty = \bigcup_{N=1}^\infty J^N\]

\noindent Finally, for each $v\in J^\infty$ we define a function $b_v$ by
\[b_{v}(z) =S_{v}\mathbf{1}(z)\]

\noindent The last theorem of this section states under what conditions these functions form an orthonormal basis for $\mathcal{H}$.

\begin{theorem} \label{onb}
If the operators $\{S_i\}_{i\in I}$ satisfy the Cuntz relations then the collection of functions $\{b_v : v\in J^\infty\}$ form an orthonormal basis for $\mathcal{H}$ and thus
\[K(z, w) = \sum_{v\in J^\infty}b_v(z)\overline{b_v(w)}\]

\end{theorem}

\section{A Family of Examples}

A key observation, due to Lemma 2.3, comes from realizing, in the context of complex dynamics, that the set $\Omega$ from Lemma 3.2 is in fact the basin of attraction at $0$ for a rational map $R$.  As shown in \cite{jet2016}, taking $t(z,w) = (z\overline{w})^\alpha$ we obtain:

\begin{theorem}
Let $R(z) = \sum_{i=1}^nc_iz^i$ such that $|c_1| < 1$ and suppose $1<\alpha\in\mathbb{N}$.  Then 
\[K(z, w) = \prod_{n=0}^\infty\left(1 + \left[R^{\circ n}(z)\overline{R^{\circ n}(w)}\right]^\alpha\right)\]
is a kernel function on the basin of attraction at $0$.
\end{theorem}

The rest of this section is devoted to showing that the results from \cite{alpay2015infinite}, given in Section 3, apply to polynomials belonging to the family 
\[\mathcal{F} = \{R_a(z) = az^{2^{n+2}} - 2az^{2^{n+1}} : a\neq 0\}.\]
\noindent In fact, the motivation for studying this family comes from  example given in \cite{alpay2015infinite}.  In particular, the authors showed that their methods worked for the polynomial $R(z) = z^4 - 2z^2$.  Using a similar approach, we will show that their methods work for any member of the family $\mathcal{F}$.\\

First note that $0$ is a super-attracting fixed point of $R_a$, for all $a\neq 0$.  Thus each $R_a$ has a basin of attraction at $0$.  Setting $t(z,w) = (z\overline{w})^{2^n}$, we may conclude from Theorem 4.1 that
\[K_a(z,w) = \prod_{n=0}^\infty\left(1 + \left[R_a^{\circ n}(z)\overline{R_a^{\circ n}(w)}\right]^{2^n}\right)\]

\noindent is a kernel function on $B_{R_a,0}$.  Denote the corresponding reproducing kernel Hilbert space by $\mathcal{H}_a$.  We now construct an orthonormal basis for $\mathcal{H}_a$.  First we show that $R_a$ satisfies $\ddagger$.  Before we begin, it is worth noting that since $t(z,w) = 1 + (z\overline{w})^{2^n}$, we have that $e_0(z) = 1$ and $e_1(z) = z^{2^n}$.

\begin{proposition}
Let $k(z,w) = 1 + (z\overline{w})^{2^{n}}$ where $n$ is a non-negative integer.  Then $R_a(z) = az^{2^{n+2}} - 2az^{2^{n+1}}$ satisfies $\ddagger$ for all $a\neq 0$.
\end{proposition}

\begin{proof}
We need to verify that the following equalities hold for any $w\in\Omega$
\[\dfrac{1}{M(w)}\sum_{R(\zeta)=w}1 = 1,\qquad\qquad \dfrac{1}{M(w)}\sum_{R(\zeta)=w}\zeta^{2^{n+1}} = 1,\qquad\qquad \dfrac{1}{M(w)}\sum_{R(\zeta)=w}\zeta^{2^n} = 0\]

\noindent where $M(w) = \text{Card}\{w\in\Omega\ : R(\zeta) = w\}<\infty$.  That $M(w) = 2^{n+2}$ follows from the fundamental theorem of algebra and that $\Omega$ is completely invariant with respect to $R$.  Thus the first equality follows.  From the substitution $u = \zeta^{2^{n+1}}$, we can determine that
\[\zeta^{2^{n+1}} = 1  \pm \dfrac{\sqrt{a^2 + aw} }{a}\]
\noindent which implies the second equality.  Taking the square root of both sides of the previous equality gives
\[\zeta^{2^n} = \pm\sqrt{1  \pm \dfrac{\sqrt{a^2 + aw} }{a}}\] 

\noindent which implies the final equality.
\end{proof}

By Theorem 3.3, the operators $S_0$ and $S_1$, defined on $\mathcal{H}_a$ by
 \[S_0(f)(z) = e_0(z)f(R_a(z)) = f(R_a(z))\] and \[S_1(f)(z) = e_1(z)f(R_a(z)) = z^{2^n}f(R_a(z)),\]
\noindent satisfy the Cuntz relations.  Therefore, by Theorem 3.4, the functions \[b_v(z) = S_v\mathbf{1}(z),\text{ where }v\in J^\infty,\] \noindent form an orthonormal basis of $\mathcal{H}_a$.  We will now study the dynamics of these basis vectors as $a$ varies.

\section{Dynamics of the family $\mathcal{F}$}

Fix $a\neq 0$, and consider the collection of basis vectors for $\mathcal{H}_a$, which we will denote $\mathcal{B}_a = \{b_{v,a}(z) :  v\in J^\infty\}$.  We want to be able to say something about how these $b_{v,a}$ look, and we will see that their appearance has very little to do with our choice of $a$.  Our two ``simplest'' basis vectors may be obtained when $v$ consists of a single component.  They are given by
\[S_0\mathbf{1}(z) = e_0(z)\mathbf{1}(R_a(z)) = \mathbf{1}(z) \text{ and }S_1\mathbf{1}(z) = e_1(z)\mathbf{1}(R_a(z)) = z^{2^n}.\]
\noindent If we let $v$ consist of two components, then we gain two basis vectors which are distinct from the previous two:
\[S_0S_1\mathbf{1}(z) = S_0(z^{2^n}) = e_0(z)(R_a(z))^{2^n} = \sum_{k=0}^{2^n}\alpha_k(a)z^{2^{2n+2}-2^{n+1}k}\]
\noindent and

\[S_1S_1\mathbf{1}(z) = S_1(z^{2^n}) = e_1(z)(R_a(z))^{2^n} = \sum_{k=0}^{2^n}\alpha_k(a)z^{2^{2n+2}-2^{n+1}k + 2^n}\]
\noindent where $\alpha_k(a) = (-2)^k\binom{2^n}{k}a^{2^n}$, for $0\leq k\leq 2^n$.  One way to describe these two basis vectors is that they are polynomials whose non-zero coefficients, $\alpha_k(a)$, are polynomials in $a$ with the following property: each $\alpha_k(a)$ has coefficients in $\mathbb{Z}$ and constant term equal to $0$.  For ease of exposition, we will say that any polynomial fitting the above description has ``good form''.  With the exception of the first two ``simple'' basis vectors, we want to show that every other basis vector has good form.  To this end we prove the following lemma:

\begin{lemma}
If $f$ is a polynomial with good form, then both $S_0 f$ and $S_1 f$ have good form as well.
\end{lemma}

\begin{proof}
If $f$ has good form, then we must have that $f(z) = \sum\limits_{i=1}^ma_iz^i$, where $a_i$ is a polynomial in $a$ with coefficients in $\mathbb{Z}$ and constant term equal to $0$.  We have that $S_0f(z) = e_0(z)f(R_a(z)) = \sum\limits_{i=1}^ma_i(R_a(z))^i$.  Observe that for each $i$, we have that

\[a_i(R_a(z))^i = \sum_{k=0}^ia_i(-2)^k\binom{i}{k}a^iz^{2^{n+1}(2i-k)},\]

\noindent which is a polynomial with good form.  Since $S_0f(z)$ is a sum of polynomials having good form, it too is a polynomial with good form.  Lastly, observe that 
\[S_1f(z) = e_1(z)f(R_a(z)) = z^{2^n}(S_0 f(z))\]
\noindent and so $S_1f(z)$ is a polynomial with good form as well.
\end{proof}

\begin{corollary}
If $v\in J^\infty$ such that $b_{v,a}(z)\neq 1$ and $b_{v,a}(z)\neq z^{2^n}$, then $b_{v,a}(z)$ is a polynomial with good form.
\end{corollary}

\begin{proof}
This follows inductively from the fact that both $S_0S_1\mathbf{1}(z)$ and $S_1S_1\mathbf{1}(z)$ are polynomials with good form.
\end{proof}

 An important observation from the above computations is the following:
 
 \begin{corollary}
 Suppose that $R_{a_1}$ and $R_{a_2}$ are polynomials in $\mathcal{F}$ and let
\[\beta_i(x)=\sum_{l=0}^{k_i}c_{i_l}x^l\]
\[\text{Then }\, b_{v,a_1}(z) = \sum_{i=0}^n\beta_i(a_1)z^i\text{ if and only if }\, b_{v,a_2}(z) = \sum_{i=0}^n\beta_i(a_2)z^i\]
 \end{corollary}
 
 \noindent Since the $b_{v,a}$ are polynomials, we can think of them as elements in the topological space of rational maps on the Riemann sphere.  In this sense, we can say that the $b_{v,a}$ vary continuously as $a$ varies.
 
 \begin{theorem}
For each $v\in J^\infty$, the map $\Gamma_v$ on $\mathbb{C}/\{0\}$ given by $a\mapsto b_{v,a}$ is continuous.
\end{theorem}

\begin{proof}
Let $a_n$ be a sequence in $\mathbb{C}/\{0\}$ converging to $a\in\mathbb{C}/\{0\}$.  Then we have
\[\lim_{n\rightarrow\infty}\Gamma_v(a_n) = \lim_{n\rightarrow\infty}b_{v,a_n} = \lim_{n\rightarrow\infty}\sum_{i=0}^m\beta_i(a_n)z^i = \sum_{i=0}^m\beta_i(a)z^i = b_{v,a} = \Gamma(a)\]

\noindent Since the coefficients of $b_v$ is given by the polynomial $\beta_i(x)$ which is continuous, we have that $\Gamma_v$ is continuous as well.
\end{proof}

\noindent The idea with this theorem is that the basis vectors $b_{v,a}$ approximate, in the sense of continuity, their ``neighboring'' basis vectors.

\newpage

\bibliographystyle{amsplain}


\end{document}